\documentclass{amsart}[12pt]
\usepackage{amsmath,amssymb,enumerate}

\usepackage{color}

\theoremstyle{plain}
\newtheorem{thm}{Theorem}[section]
\newtheorem{lem}[thm]{Lemma}
\theoremstyle{definition}
\newtheorem{defi}[thm]{Definition}

\newtheorem{remark}[thm]{Remark}

\newtoks\by
\newtoks\paper
\newtoks\book
\newtoks\jour
\newtoks\yr
\newtoks\pages
\newtoks\vol
\newtoks\publ
\def\ota{{\hbox\vol{???}}}
\def\cLear{\by=\ota\paper=\ota\book=\ota\jour=\ota\yr=\ota
\pages=\ota\vol=\ota\publ=\ota}
\def\endpaper{\the\by, \the\paper.
{\it\the\jour\/} {\bf \the\vol} (\the\yr), \the\pages.\cLear}
\def\endbook{\the\by, {\it\the\book}. \the\publ.\cLear}
\def\endprep{\the\by, \the\paper. \the\jour.\cLear}
\def\name#1#2{#1 #2}
\def\et{ and }
\def\supp{\textup{supp }}

\def\id{\textup{id}}
\def\spa{\textup{span}}
\def\sign{\textup{sign}}
\def\be{\begin{equation}}
\def\ee{\end{equation}}

\numberwithin{equation}{section} \headheight=12pt
\begin{document}
\title[Embeddings between Lorenz \dots]{Embedding between Lebesgue and weak Lebesgue sequence spaces is strictly singular}

\author[J.Lang]{J. Lang }
\address{Department of Mathematics, The Ohio State University,
231 West 18th Avenue, Columbus, OH 43210, USA}
\email{lang@math.osu.edu}
\author[A.Nekvinda]{A. Nekvinda}
\address{Department of Mathematics\\ Faculty of Civil Engineereng\\
       Czech Technical University\\Th\' akurova 7\\16629 Prague 6\\
       Czech Republic}
\email{nales@mat.fsv.cvut.cz}
\date{}
\thanks{The second author was supported by the grant P201-18-00580S of the Grant Agency of the Czech Republic}
%\thanks{E-mail of the corresponding author: \texttt{??????}}

\maketitle
\begin{abstract}
Given $1< p < q < \infty$ it is well know that the natural embedding of Lebesgue sequence spaces $\ell_p \hookrightarrow \ell_q$ is strictly singular. In this paper we extend this classical results and show that even the natural non-compact embedding  between Lebesgue and weak Lebesgue sequence spaces $\ell_{p}\hookrightarrow \ell_{p,\infty}$ is strictly singular. 
\end{abstract}
\section{Introduction}
It is true generally acknowledge that among all bounded operators acting on Banach spaces the compact operators hold a quite unique position since they play an essential role in many different areas of mathematics. Then  also  operators, which are in some ``sense" close to compact operators, deserve detailed study. Among all classes of  non-compact operators which are ``close" to compact maps the special position is occupied  by strictly singular and by finitely strictly singular operators (for definition see Sec. 2).

Let us mention a couple of examples highlighting importance of strictly singular operators.
It is well know that Fredholm operators are invariant when  perturbed by strictly singular operators
(i.e. if $T$ is Fredholm and $S$ is strictly singular then $T+S$ is Fredholm, see \cite[Therem 4.63]{AA book}). And that the limiting case of  Fourier transformation $\mathcal{F}$ considered as an operator from $L^1$ into $L^\infty$, which  is obviously non-compact, is strictly singular (see \cite{Pl}).
%	{\color{red}  In \cite{CMMM} was proved that strictly singular and strictly co-singular operators could be interpolated. If $T:A_0 \rightarrow B_0$   is strictly singular and $T:A_1 \rightarrow B_1$  is bounded, the interpolated operator $T:(A_0 ,A_1 )_{\theta,q}\rightarrow (B_0 ,B_1 )_{\theta,q}$   is not strictly singular in general, however, it is strictly singular if $A=A_0 =A_1$. Similar assertion holds for so called strictly co-singular operators. Let us note that the similar behavior for interpolation  is  observed with entropy numbers.
%	}
%
%Useful information about strict singularity or finite strict singularity of an operator $T:X\to Y$ can be obtained from the behavior so called Bernstein numbers (or Bernstein widths) defined by
%$$
%b_n(T)=\sup_{ E\subset X,\dim(E)=n } \inf_{f\in E,\|f\|_X=1}  \|T(f)\|_{Y}.
%$$
%It is possible to see that $T$ is finitely strictly singular if and only if $b_n(T)\rightarrow 0$.
%
%{\color{red} Let us consider the Volterra type operator
%$V(f)(x)=\int_0^x f(t)dt$ defined on interval $(0,1)$ as a map $V:L^1([0,1])\rightarrow \mathcal{C}([0,1])$.
%Obviously, in this set-up,  $V$ is bounded.
% In the paper \cite{PL} and \cite{L} the exact values for strict s-numbers  were evaluated and it was observed that that the approximation numbers $a_n(V)= 1/2$ for all $n\ge 2$ and  $b_n(V)=1/(2n-1)$ which means that $V$ is non-compact but finitely strictly singular map.}
Also when we consider limiting Sobolev embedding $E_d$ (here by limiting Sobolev embedding we understand Sobolev embedding for which is impossible to "significantly" increase the source space or decrease the target space without loosing the boundedness):
\begin{align} \label{eq E_d}
	&E_d:W^{1}_{0}L^{d,1} ((0,1)^d )\hookrightarrow C((0,1)^d ),
\end{align}
where  $W^{1}_{0}L^{d,1} ((0,1)^d )$ denote a space of all functions $u$ for which $|\nabla u| $ belongs to Lorentz space $L^{d,1}$  and $u$ has a zero trace. It was observed, from the behavior of strict s-numbers for $E_d$, in \cite{LM} that $E_d$, which is non-compact, is finitely strictly singular and then strictly singular map. 

 %And this is the largest Sobolev space embedded into continuous functions on $(0,1)^d$.)
%In the one dimensional case ($d=1$) it was proved  for approximation numbers
%\begin{align*}
%	&a_n(E_1)=\frac12, \quad \mbox{when } n\ge 2,
%\end{align*}
%and for the Bernstein numbers
%\begin{align*}
%	&b_n(E_1)=\frac{1}{2n}, \qquad \mbox{ for } n\ge 1.
%\end{align*}
%In the higher dimension ($d\ge 2$) it was shown, among others, that
%\begin{align*}
%	&a_n(E_d) \asymp 1, \quad \mbox{for } n\ge 1,
%\end{align*}
%and that
%\begin{align*}
%	&b_n(E_d) \asymp {n^{-1/d}}, \qquad \mbox{ for } n\ge 1.
%\end{align*}

%{\color{red},  as in the case of Volterra operator,}

From the above examples arises a natural question: Are all ``limiting'' Sobolev embeddings on bounded domain  strictly singular?% or finitely strictly singular?

In many cases, as in (\ref{eq E_d}), the limiting Sobolev embeddings have the optimal source or the optimal target space equal to Lorentz space. In order to be able study the above question we need to have more accurate   information about relation between Lorentz spaces, for instance if the natural embedding between Lebesgue sequence and Lorentz sequence spaces
 \begin{equation} \label{Lorentz Embd}
 I:\ell_{p} \to\ell_{p,\infty},
 \end{equation}
 is strictly singular.% or even finitely strictly singular. 
 \ This question, which can be also considered as a natural generalization of of well known fact that for $1<p<q<\infty$ the non-compact embedding $\ell_p \hookrightarrow \ell_q$ is strictly singular (see \cite[Therem 4.58]{AA book}), is the focus of our paper and we will prove that the non-compact  embedding (\ref{Lorentz Embd}) is  strictly singular.

 The paper is structured as follows. In Sect. 2, we recall  definitions
 and we collect all necessary later-needed material and technical lemmas. In Sect. 3 is proved that embedding $\ell_{p} \hookrightarrow\ell_{p,\infty}$ is strictly singular.

\section{preliminaries}

In this section we recall definitions, notations and some technical lemmas needed in Section 3.
We start by recalling the definition  of strictly singular operators.
\begin{defi}
A bounded operator $T:X \to Y$ between Banach spaces is said to be strictly singular if there is no infinite dimensional closed subspace $Z$ of $X$ such that $T:Z \to T(Z)$, the restriction of $T$ to $Z$, is an isomorphism.	
\end{defi}

%\begin{defi}
%	An operator $T$ from a Banach space $X$ into a Banach space $Y$
%	is finitely strictly singular if: for every $\varepsilon> 0$, there exists $n_{\varepsilon}\ge 1$ such that every
%	subspace $E$ of $X$ with dimension greater that $n_{\varepsilon}$, there exists $x$ in the unit sphere
%	of $E$ such that $\|T(x)\|_Y\le \varepsilon$.
%\end{defi}

See \cite[section 4.5]{AA book} for more information about strictly singular operators. Note that it is not difficult to see that each compact operator is strictly singular. Also it is worthy to mention that there exists a class of operator called {\bf finitely strictly singular operators} which lays between compact and strictly singular operators. These are operators for which Bernstein numbers are vanishing, $b_n(T) \to 0$.

By $\#(F)$ we denote the number of elements of a finite set $F$.  For a sequence $u=(u_1,u_2,\dots)$ we define the corresponding modulus sequence $|u|=(|u_1|, |u_2|, \dots)$. We say that $|u|\le |v|$ if $|u_i|\le |v_i|$ for each $i\in\mathbb{N}$.

Now we introduce Banach function spaces which we will need later.
\begin{defi}
Let $\mathcal{S}$ be a a set of all sequences of real numbers and $\|.\|:\mathcal{S}\rightarrow[0,\infty]$. Assume that $\|.\|$ satisfies  for all $u,v\in \mathcal{S}$ and $\alpha\in\mathbb{R}$ we have
\begin{enumerate}[\rm(i)]
\item  $\|u+v\|\le \|u\|+\|v\|$,
\item $\|\alpha u\|=|\alpha|\ \|u\|$,
\item $\|u\|\ge 0$ and $\|u\|= 0$ if and only if $u=0$,
\item $\|u\|=\|\ |u|\ \|$,
\item if $|u|\le |v|$ then  $\|u\|\le \|v\|$,
\item if $0\le u_n\nearrow u$ then $\|u_n\|\nearrow\|u\|$,
\item if $\#\{i;u_i\neq 0\}<\infty$ then $\|u\|<\infty$.
\end{enumerate}
Define $X:=\{u;\|u\|<\infty\}$. Then we call $X$ a sequence Banach function space.
\end{defi}
%By an analogous way we could define a quasi-Banach function space of functions on a domain $\Omega$.
Remark that each Banach function space is complete (for details see \cite{BS}, Theorem 1.6).

We repeated the definition of strictly singular operators on Banach function spaces by the following alternative definition:
\begin{defi}
Let $X,Y$ be Banach spaces and assume that $T:X\rightarrow Y$ be a linear bounded operator. We say that $T$ is strictly singular operator if
\begin{align*}
&\inf\{\|Tx\|_Y;\|x\|_X=1, x\in Z\}=0
\end{align*}
for each infinite dimensional subspace $Z\subset X$.
\end{defi}

\begin{defi}
Given a sequence $a=(a(1),a(2),\dots)\in c_0$ we set for $\lambda>0$
\begin{align*}
&\mu_a(\lambda)=\#\{i;|a(i)|>\lambda\}
\end{align*}
and
\begin{align*}
&a^*(j)=\min\{\lambda>0;\mu_a(\lambda)\le j\}.
\end{align*}
We will call $a^*=(a^*(1),a^*(2),\dots)$ as a non-increasing rearrangement of $a$.
\end{defi}
For a sequence $a=(a(1),a(2),\dots)\in c_0$ we define
\begin{align*}
&\supp a=\{j\in \mathbb{N}; a(j)\neq 0\}.
\end{align*}

\begin{defi}
Given a sequence $u=(u(1),u(2),\dots)\in c_0$ with $\supp u=\{n_1,n_2,...,n_k\}\subset\mathbb{N}$ and $n_1<n_2<\dots n_k$. Define a non-increasing rearrangement $u^{\diamond}$ of $u$ with respect to $\supp u$ by
\begin{align*}
&\begin{cases}
u^{\diamond}(n_j)=u^*(j)&\ j\in\{1,2,\dots,k\},\\
u^{\diamond}(i)=0&\ i\notin\{n_1,n_2,...,n_k\}.
\end{cases}
\end{align*}

\end{defi}
\begin{remark}
If $\supp u:=\{n+1,n+2,\dots,m\}$ then
\begin{align}
&u^{\diamond}(j)=u^*(j-n).\label{vjrvjvjvvvj}
\end{align}
\end{remark}

In the next we recall the definition of sequence spaces.
\begin{defi}
Let $1< p<\infty $. We set for a sequence $u$  norms:
\begin{align*}
&\|u\|_{p}=\Big(\sum\limits_{j=1}^\infty (u^*(j))^p \Big)^{1/p},\\
&\|u\|_{p,\infty}=\sup\ j^{1/p}\ u^*(j),
\end{align*}
and we define Lebesgue space $\ell_p$ and weak Lebesgue space $\ell_{p,\infty}$ as collections of all sequences $u$ for which the norms $\|u\|_{p}$ and $\|u\|_{p,\infty}$ are finite, respectively.
\end{defi}

Unfortunately, $\|.\|_{p,\infty}$ do not satisfy the triangle inequality (i). But it is well-known (see \cite{BS}, Lemma 4.5) that for $1<p<\infty$ there is a norm $\|.\|_{(p,\infty)}$ equivalent to $\|u\|_{p,\infty}$ which satisfies (i), ..., (vii). So, we can assume that $\|.\|_{p,\infty}$ satisfies (i) and we will mentioned the next well known lemma without proof. 
 \begin{lem} \label{vnor}
Let $1<p<\infty$. Then both spaces $\ell_{p}$ and $\ell_{p,\infty}$ are Banach function spaces and $\ell^{p}\hookrightarrow \ell^{p,\infty}$
 \end{lem}
Given $u\in \ell_{p}\ (u\in\ell_{p,\infty})$ we will write $u(i)$ for the value of $u$ at the index $i$.

%\begin{lem}
%Let $0<p<\infty, 0<q\le \infty$.
%The space $\ell_{p}$ is a quasi-Banach function space.
%\end{lem}
%\begin{proof}
%As in \cite{BS} (see (1.16) in Proposition 1.7) we can prove
%\begin{align*}
%&(u+v)^*(i+j)\le u^*(i)+v^*(j).
%\end{align*}
%Split the sum
%\begin{align*}
%&\sum\limits_{j=1}^\infty j^{q/p-1}(u+v)^*(j))^q
%\end{align*}
%into two sums, the first one is  over odd numbers, the second one is over even numbers. For both sums we can easily prove the quasi-triangle inequality. The other properties are easy.
%\end{proof}

%\begin{lem}\label{vjifdjkvfjo}
%Let $0< p<\infty,0< q<\infty$. Then we have for all $n\in\mathbb{N}$
%\begin{align*}
%&\Big(\sum_{j=1}^n j^{q/p-1}\Big)^{1/q}\approx n^{1/p}.
%\end{align*}
%\end{lem}
%\begin{proof}
%For all $n$ we have
%\begin{align*}
%&\sum_{j=1}^n j^{q/p-1}\approx \int_0^{n} t^{q/p-1}\ dt=\frac{p}{q}n^{q/p}\approx  n^{q/p}.
%\end{align*}
%\end{proof}

 We denote by $D_{p}$ the norm of this embedding, i.e.
\begin{align}
&D_p=\sup\{\|a\|_{p,\infty};\|a\|_{p}\le 1\} .\label{dcehoiefroir}
\end{align}

For a sequence $b=(b(1),b(2),\dots)$ and $m\in\mathbb{N}$
we define
\begin{align*}
	&P_m(b)=(b(1),b(2),\dots,b(m),0,0,\dots)\\
	&R_m(b)=b-P_m b=(0,0,\dots,0,b({m+1}),b({m+2}),\dots).
\end{align*}

\begin{defi}
	Let $X$ be a Banach function space of sequences. We say that $u\in X$ has an absolutely continuous norm in $X$, written $u\in X_a$, if 
\begin{align*}
	&\lim_{m\rightarrow\infty}\|R_m(u)\|_X= 0.
\end{align*}
 We say that $X$ has an absolutely continuous norm if $X_a=X$.
	
\end{defi}

We left proofs for the next two lemmas to the reader.

\begin{lem}
Let $1< p<\infty$. Then $\ell_{p}$ has an absolutely continuous norm, i.e.
\begin{align*}
&\lim_{m\rightarrow\infty} \|R_m (u)\|_p=0\ \text{ for each }\ u\in \ell_p.
\end{align*}
\end{lem}

%\begin{lem}\label{sdcdlvhndvnhvki}
%Let $0< p<\infty,0< q<\infty$ and $\alpha>0$. Assume $v_j\in\ell_{p}$ have pairwise disjoint supports and $\|v_j\|_{p}\ge \alpha$.
%Then
%\begin{align*}
%&\lim_{k\rightarrow\infty}\Big\|\sum_{j=1}^k v_j\Big\|_{p}= \infty.
%\end{align*}
%\end{lem}

\begin{lem}\label{fvfvjpfvj}
Let $1< p<\infty$. Assume that $X\subset\ell_{p}$ be a subspace, $\dim X=\infty$.
Set
\begin{align*}
&X_n=\{a\in X; a(1)=a(2)=\dots=a(n)=0\}.
\end{align*}
Then $\dim X_n=\infty$.
\end{lem}

% Assume $\dim X_n:=k<\infty$. Let $b_1,b_2,\dots, b_k\in X$ be a basis of $X_n$.
%
%Consider now a set $Z_n=\{P_n x; x\in X\}$.  Clearly, $Z_n$ is a subspace of $X$ and $\dim Z_n<\infty$.
%Then there exist sequences $y_1,y_2,\dots,y_s\in X$ such that $z_i:=P_n y_i, i=1,2,\dots,s$ establish a base in $Z_n$.
%
%Choose now $x\in X$. Then
%\begin{align*}
%&P_n x=\sum_{j=1}^s \alpha_j z_j.
%\end{align*}
%It gives
%\begin{align*}
%& x(r)=\sum_{j=1}^s \alpha_j y_j(r),\ \ \ r=1,2,\dots,n.
%\end{align*}
%Consequently $x-\sum_{j=1}^s \alpha_j y_j\in X_n$
%and so
%$x-\sum_{j=1}^s \alpha_j y_j=\sum_{l=1}^k \beta_l b_l$.
%It gives
%\begin{align*}
%& x=\sum_{j=1}^s \alpha_j y_j+\sum_{l=1}^k \beta_l b_l
%\end{align*}
%and $\dim X\le s+k $. This contradicts our assumption.

\begin{lem}\label{adkcvhovfj}
Suppose $1\le p<\infty$. Let $X\subset \ell_{p}$ be a closed subspace with $\dim X=\infty$.
Assume $n, N\in\mathbb{N}$ and $\varepsilon>0,1> \delta>0$. Then there exists $m\in\mathbb{N}$ and $u\in X_n$ such that denoting $v:=P_m u$, $w:=R_m u$ we have
\begin{align}
&\|u\|_{p}=1,\label{evklfkobb}\\
&m> 2n,\ m\ge N\label{dcvwvjwvjj}\\
&\supp v\subset\{n+1,n+2,\dots,m\},\label{dlkdlkvdkvk}\\
&|v(j)|\le \varepsilon\text{ for all }j,\label{skcnlvwig}\\
&1-\delta\le \|v\|_{p}\le1,\label{odcjvjevjvj}\\
&\|w\|_{p}\le \delta.\label{spvdjpbjpojb}
\end{align}
\end{lem}
\begin{proof}Set $n_0:=n$ and
construct, by induction, sequences $n_0< n_1<n_2<\dots$ and $u_i\in X$ such that for $v_i:=P_{n_{i}}u_i$, $w_i:=R_{n_{i}}u_i$ we have
\begin{align}
&\supp v_i\subset\{n_{i-1}+1,n_{i-1}+2,\dots,n_{i}\},\label{dkdcdjpvdpvv}\\
&1-\frac{\delta}{2^{i}}\le \|v_i\|_{p}\le 1.\label{wdkvjwpvjwvppj}\\
&\|w_i\|_{p}\le \frac{\delta}{2^{i}}.\label{dlkvdnvkdjvdkv}
\end{align}

Since $\dim X_n=\infty$ we can find $u_1\in X_n$ with $\|u_1\|_{p}=1$. Take $n_1>n$ such that $\|R_{n_1}u_1\|_{p}\le \delta/2$. Denote $v_1:=P_{n_1}u_1$, $w_1:=R_{n_1}u_1$. Clearly, $\supp v_1\subset\{n+1,n+2,\dots,n_1\}$ and
\begin{align*}
&1\ge \|v_1\|_{p}\ge \|u_1\|_{p}-\|w_1\|_{p}\ge 1-\frac{\delta}{2}.
\end{align*}

Suppose that we have constructed $n_0< n_1<n_2<\dots<n_k$, $u_1,u_2,\dots,u_k\in X$ and appropriate functions $v_1,v_2,\dots,v_k$ satisfying  \eqref{dkdcdjpvdpvv} and \eqref{wdkvjwpvjwvppj}.
Since $\dim X_{n_k}=\infty$ we are able to find $u_{k+1}\in X_{n_k}$ with $\|u_{k+1}\|_{p}=1$. It is easy to see that we can take an index $n_{k+1}> n_k$ such that
$\|R_{n_{k+1}}u_{k+1}\|_{p}\le \frac{\delta}{2^{k+1}}$. Set $w_{k+1}=R_{n_{k+1}}u_{k+1}$, $v_{k+1}=P_{n_{k+1}}u_{k+1}$. Consequently
\begin{align*}
&1\ge \|v_{k+1}\|_{p}\ge \|u_{k+1}\|_{p}-\|w_{k+1}\|_{p}\ge 1-\frac{\delta}{2^{k+1}}.
\end{align*}
Moreover $\supp v_{k+1}\subset\{n_{k}+1,n_{k}+2,\dots,n_{k+1}\}$.

Now, consider sequences
\begin{align*}
&y_k:=\sum_{j=1}^k u_j,\ \  \ s_k:=\|y_k\|_{p}.
\end{align*}
%By \eqref{sdcdiovnhdo}, \eqref{dovdovio} and \eqref{dlkvdnvkdjvdkv}
We can write
\begin{align*}
&s_k=\Big\|\sum_{j=1}^k u_j\Big\|_{p}=\Big\|\sum_{j=1}^k v_j+\sum_{j=1}^k w_j\Big\|_{p}\ge \Big\|\sum_{j=1}^k v_j\Big\|_{p}-\Big\|\sum_{j=1}^k w_j\Big\|_{p}\\
&\ge \Big\|\sum_{j=1}^k v_j\Big\|_{p}-\sum_{j=1}^k \|w_j\|_{p}\ge \Big\|\sum_{j=1}^k v_j\Big\|_{p}-\sum_{j=1}^k \frac{\delta}{2^{j}}\ge \Big\|\sum_{j=1}^k v_j\Big\|_{p}-\delta.
\end{align*}
Since by \eqref{wdkvjwpvjwvppj} we obtain
\begin{align*}
& \|v_i\|_{p}\ge1-\frac{\delta}{2^{i}}\ge 1-\delta
\end{align*}
and $v_j$ have pairwise disjoint supports by \eqref{dlkdlkvdkvk}, then
\begin{align*}
&s_k=\Big\|\sum_{j=1}^k v_j\Big\|_{p}=\Big(\sum_{j=1}^k \|v_j\|_{p}^p\Big)^{1/p}\ge\Big(\sum_{j=1}^k(1-\delta)^p\Big)^{1/p}\ge (1-\delta)k^{1/p}
\end{align*}
and consequently $s_k\nearrow\infty$.

Then we can choose $m$ large enough such that
\begin{align}
&m> 2n,\ m\ge N,\ \frac{1}{s_m}\le \varepsilon\label{wdvfhnrogronhi}
\end{align}
and set
\begin{align*}
&u=\frac{1}{s_m}\sum_{j=1}^m u_j.
\end{align*}

It is seen from the definition of $s_m$ that
\begin{align*}
&\|u\|_{p}= 1,
\end{align*}
which proves \eqref{evklfkobb}.

Clearly, condition \eqref{dcvwvjwvjj} is satisfied. By the definition of $v=P_m u$ we obtain directly
\begin{align*}
&\supp v\subset\{n_0+1,n_0+2,\dots,m\}=\{n+1,n+2,\dots,m\}
\end{align*}
which proves \eqref{dlkdlkvdkvk}.

Since $\|v_k\|_p\le 1$ we have for each $s\in\mathbb{N}$
\begin{align*}
&|u_k(s)|\le 1.
\end{align*}
%Remark that $\supp v_j\subset\{n_{i-1}+1,n_{i-1}+2,\dots,n_{i}\}$ and so $v_j$ have pairwise disjoint supports. Thus
Clearly, using that $u_j$ have pairwise disjoint supports, we have for each $s$
\begin{align*}
&|v(s)|\le |u(s)|\le \frac{1}{s_m}\sum_{j=1}^m |u_j(s)|\le \frac{1}{s_m}\overset{\eqref{wdvfhnrogronhi}}{\le} \varepsilon
\end{align*}
which proves \eqref{skcnlvwig}.

At last,
\begin{align*}
&w=R_mu=R_m\Big(\frac{1}{s_m}\sum_{j=1}^m u_j\Big)=R_m\Big(\frac{1}{s_m}\sum_{j=1}^m v_j+\frac{1}{s_m}\sum_{j=1}^m w_j\Big)\\
&\frac{1}{s_m}\sum_{j=1}^m R_m( v_j)+\frac{1}{s_m}\sum_{j=1}^m R_m( w_j)=\frac{1}{s_m}\sum_{j=1}^m R_m( w_j).
\end{align*}
Thus
\begin{align*}
&\|w\|_{p}\le\frac{1}{s_m}\sum_{j=1}^m  \|R_m( w_j)\|_{p}\le\frac{1}{s_m}\sum_{j=1}^m  \|w_j\|_{p}\\
&\overset{\eqref{dlkvdnvkdjvdkv}}{\le}\frac{1}{s_m}\sum_{j=1}^m  \frac{\delta}{2^{i}}\le\frac{1}{s_m}\sum_{j=1}^m \frac{\delta}{2^j}\le\frac{\delta}{s_m}\le \delta
\end{align*}
which proves \eqref{spvdjpbjpojb}.

Finally,   \eqref{odcjvjevjvj} follows directly from
\begin{align*}
&1\ge \|v\|_{p}\ge\|u\|_{p}-\|w\|_{p}\ge 1-\delta
\end{align*}
which concludes the proof.
\end{proof}

\section{Main theorem}

\begin{thm}
Let $1< p<\infty$. Then the embedding $\ell_{p}\hookrightarrow\ell_{p,\infty}$ is strictly singular.
\end{thm}

\begin{proof}
Having a sequence $0=n_0<n_1<n_2<\dots$ and $u_k\in X_{n_{k-1}}$, $k\ge 1$, we denote
\begin{align*}
&v_k=P_{n_k} u_k,\ w_k=R_{n_k} u_k,\\
&I_k=\{n_{k-1}+1,n_{k-1}+2,\dots,n_k\},\\
&b_k=\min\{|v_k(j)|;v_{k}(j)\neq 0,j\in I_k\}.
\end{align*}

Choose $0<\delta<1$.
We will construct by mathematical induction a sequence $0=n_0<n_1<n_2<\dots$ and $u_k\in X_{n_{k-1}}$, $k\ge 1$,
such that
\begin{align}
&\|u_k\|_{p}= 1,\label{wdkjvlov}\\
&2n_{k-1}<  n_k.\label{dklvhdvhjvk}\\
&\supp v_k\subset I_k,\label{efkvjefbvjop}\\
&|v_{k+1}(j)|\le \min\Big\{ b_k,\frac{1}{(n_k-n_{k-1})^{1/p}}\Big\},\label{efklbvkpbnmkb}\\
&\frac{1}{(n_{k+1}-n_{k})^{1/p}}\le b_k,\label{vfjkovnhhk}\\
&1-\delta\le \|v_k\|_{p}\le1,\label{felkvnjfbnb}\\
&\|w_k\|_{p}\le \frac{\delta}{2^{k}}.\label{klvjnfovn}
\end{align}
Consider first $k=1$. Find $u_1\in X_0=X$ with $\|u_1\|_{p}=1$. Then we can choose $n_1>n_0$ such that $\|w_1\|_{p}\le \frac{\delta}{2}$. Consequently
\begin{align*}
&\|v_1\|_{p}\overset{}{\ge}\|u_1\|_{p}-\|w_1\|_{p}\ge 1-\frac{\delta}{2} \ge 1-\delta.
\end{align*}
It is easy now to verify conditions \eqref{wdkjvlov} -- \eqref{klvjnfovn}.

Assume that we have constructed $0=n_0<n_1<n_2<\dots<n_k$ and $u_1,u_2,\dots,u_k$, $u_i\in X_{n_{i-1}}$ satisfying \eqref{wdkjvlov} -- \eqref{klvjnfovn}. Consider a space $X_{n_k}$. Choose $\frac{\delta}{2^{k+1}}$ instead of $\delta$ in Lemma \ref{adkcvhovfj} and set
\begin{align*}
&\varepsilon:=\min\Big\{b_k,\frac{1}{(n_{k}-n_{k-1})^{1/p}}\Big\}.
\end{align*}
Find $N>n_k$ such that
\begin{align*}
&\frac{1}{(N-n_{k})^{1/p}}\le b_k.
\end{align*}
By  Lemma \ref{adkcvhovfj} there exist $u\in X_{n_k}$, $\|u\|_{p}=1$ and $m\ge N$, $m\ge 2n_k$ such that for $v=P_m u$, $w=R_m u$ we have
\begin{align*}
&\supp v\subset\{n_k+1,n_k+2,\dots,m\},\\
&|v(j)|\le \varepsilon,\\
&1-\delta\le \|v\|_{p}\le 1,\\
&\|w\|_{p}\le \frac{\delta}{2^{k+1}}.
\end{align*}
Now, it suffices to choose $ n_{k+1}=m $ and $u_{k+1}=u$.
%%%%%%%%%%%%%%%%%%%%%%%%%%%%%%%%%%%%%%%%%%%%%%%%%%%%%%%%%%%%%%%%

Set now
\begin{align*}
&z_N=\sum_{j=1}^N u_j\in X.
\end{align*}

Then we can write
\begin{align}
&\|z_N\|_{p}=\Big\|\sum_{j=1}^N u_j\Big\|_{p}=\Big\|\sum_{j=1}^N v_j+\sum_{j=1}^N w_j\Big\|_{p}\label{wdoicjoidwvcjiw}\\
&\overset{}{\ge} \Big\|\sum_{j=1}^N v_j\Big\|_{p}-\Big\|\sum_{j=1}^N w_j\Big\|_{p}\overset{}{\ge} \Big\|\sum_{j=1}^N v_j\Big\|_{p}-\sum_{j=1}^N \|w_j\|_{p}\nonumber\\
&\overset{\eqref{klvjnfovn}}{\ge}   \Big\|\sum_{j=1}^N v_j\Big\|_{p}-\sum_{j=1}^k \frac{\delta}{2^{j}}\ge \Big\|\sum_{j=1}^N v_j\Big\|_{p}-\delta.\nonumber
\end{align}
By \eqref{felkvnjfbnb} we obtain
\begin{align*}
& \|v_i\|_{p}\ge1-\delta,
\end{align*}
and due \eqref{efkvjefbvjop} we have that  $v_j$ have pairwise disjoint supports and then
\begin{align*}
&\Big\|\sum_{j=1}^N v_j\Big\|_{p}=\Big(\sum_{j=1}^N \|v_j\|_p^p\Big)^{1/p}\ge (1-\delta)N^{1/p}
\end{align*}
which gives with \eqref{wdoicjoidwvcjiw} that
\begin{align*}
&\|z_N\|_{p}\rightarrow\infty.
\end{align*}

Now we try to estimate $\|z_N\|_{p,\infty}$. 
Clearly
\begin{align}
&\|z_N\|_{p,\infty}=\Big\|\sum_{j=1}^N u_j\Big\|_{p,\infty}=\Big\|\sum_{j=1}^N v_j+\sum_{j=1}^N w_j\Big\|_{p,\infty}\label{sdoicvhdiovnh}\\
&\le \Big\|\sum_{j=1}^N v_j\Big\|_{p,\infty}+\Big\|\sum_{j=1}^N w_j\Big\|_{p,\infty}
                    \le \Big\|\sum_{j=1}^N v_j\Big\|_{p,\infty}+\sum_{j=1}^N \|w_j\|_{p,\infty}\nonumber\\
&\overset{\eqref{klvjnfovn}}{\le}
\Big\|\sum_{j=1}^N v_j\Big\|_{p,\infty}+\sum_{j=1}^N \frac{\delta}{2^{j}}\le\Big\|\sum_{j=1}^N v_j\Big\|_{p,\infty}+\delta.\nonumber
\end{align}

It remains to estimate $\Big\|\sum_{j=1}^N v_j\Big\|_{p,\infty}$.
Denote
\begin{align*}
&A_k=\{j\in I_k, v_k(j)=0\}.
\end{align*}
Set
\begin{align*}
&\widetilde{v}_{k}(j)=|v_k(j)|+\frac{1}{(n_k-n_{k-1})^{1/p}}\chi_{A_k}(j)
\end{align*}
and
\begin{align*}
&\widetilde{z}_N=\sum_{k=1}^N \widetilde{v}_{k}.
\end{align*}
Since  $|{v}_{k}(j)|\le\widetilde{v}_{k}(j)$ we have
\begin{align}
&\|z_N\|_{p,\infty}\overset{\eqref{sdoicvhdiovnh}}{\le}2^{\frac{1}{p}}\Big(\Big\|\sum_{j=1}^N v_j\Big\|_{p,\infty}+\delta\Big)\le 2^{\frac{1}{p}}\Big(\Big\|\sum_{j=1}^N |v_j|\Big\|_{p,\infty}+\delta\Big)\label{evhfovho}\\
&\le 2^{\frac{1}{p}}\Big(\Big\|\sum_{j=1}^N \widetilde{v}_j\Big\|_{p,\infty}+\delta\Big)=(\|\widetilde{z}_N\|_{p,\infty}+\delta).\nonumber
\end{align}
Take $i\in I_k$ and $j\in I_{k+1}$. Assume first $v_k(i)\neq 0$. Then
\begin{align*}
&\widetilde{v}_k(i)=|v_k(i)|\ge b_k\overset{\eqref{efklbvkpbnmkb}}{\ge}|v_{k+1}(j)|
\end{align*}
and also
\begin{align*}
&\widetilde{v}_k(i)\ge b_k\overset{\eqref{efklbvkpbnmkb}}{\ge}\frac{1}{(n_{k+1}-n_{k})^{1/p}}
\end{align*}
and so
\begin{align*}
&\widetilde{v}_k(i)\ge |v_{k+1}(j)|+\frac{1}{(n_{k+1}-n_{k})^{1/p}}\chi_{A_{k+1}}(j)=\widetilde{v}_{k+1}(j).
\end{align*}
 If $v_k(i)=0$ then
\begin{align*}
&\widetilde{v}_k(i)=\frac{1}{(n_{k}-n_{k-1})^{1/p}}\overset{\eqref{efklbvkpbnmkb}}{\ge}|v_{k+1}(j)|.
\end{align*}
Further by \eqref{dklvhdvhjvk} we have $n_{k+1}\ge 2n_k\ge 2n_k-n_{k-1}$ which implies
\begin{align*}
&\frac{1}{(n_{k}-n_{k-1})^{1/p}}\ge \frac{1}{(n_{k+1}-n_{k})^{1/p}}
\end{align*}
and so
\begin{align*}
&\widetilde{v}_k(i)=\frac{1}{(n_{k}-n_{k-1})^{1/p}}{\ge}\frac{1}{(n_{k+1}-n_{k})^{1/p}}.
\end{align*}
Consequently
\begin{align*}
&\widetilde{v}_k(i)\ge |v_{k+1}(j)|+\frac{1}{(n_{k+1}-n_{k})^{1/p}}\chi_{A_{k+1}}(j)=\widetilde{v}_{k+1}(j).
\end{align*}

We have proved
\begin{align}
&|\widetilde{v}_{k+1}(j)|\le |\widetilde{v}_k(i)| \ \ i\in I_k, j\in I_{k+1}.\label{sdjidjvj}
\end{align}

Fix $j\in\mathbb{N}$. Since $I_k$ have pairwise disjoint supports and $\mathbb{N}=\bigcup_{k=1}^\infty$ there exists a unique $k\in\mathbb{N}$ such that $j\in I_k$.

If $k>N$ then $\widetilde{z}_N(j)=0$ and since $|\widetilde{z}_N(i)|>0$ for $i\le n_N$ we obtain $(\widetilde{z}_N)^*(j)=0$ and so
\begin{align}
&j^{1/p}(\widetilde{z}_N)^*(j)=0, \ \ j\ge n_N+1.\label{wdjchcviohivh}
\end{align}
Let $k\le N$. Then $n_{k-1}+1\le j\le n_k$ and by \eqref{sdjidjvj} there is $n_{k-1}+1\le i\le n_k$ such that $(\widetilde{z}_N)^*(j)=(\widetilde{v}_{k})^*(j-n_{k-1})=\widetilde{v}_{k}(i)$. 

We have two possibilities. Either {\bf(a)} $n_{k-1}+1\le j\le 2n_{k-1}$ or {\bf(b)} $2n_{k-1}< j\le n_k$.

{\bf(a)} If $n_{k-1}+1\le j\le 2n_{k-1}$ we have
\begin{align}
&j^{1/p}(\widetilde{z}_N)^*(j)=j^{1/p}(\widetilde{v}_k)^*(j-n_{k-1})\le (2n_{k-1})^{1/p}\widetilde{v}_k(i)\label{evoriovg}\\
&\overset{\eqref{efklbvkpbnmkb}}{\le} \frac{(2n_{k-1})^{1/p}}{(n_{k-1}-n_{k-2})^{1/p}}\overset{\eqref{dklvhdvhjvk}}{\le} 4^{1/p}\le 2^{1+1/p}.\nonumber
\end{align}

{\bf(b)} If $2n_{k-1}< j\le n_k$ we obtain by Proposition \ref{vnor}
\begin{align*}
&j^{1/p}(\widetilde{z}_N)^*(j)=j^{1/p}(\widetilde{v}_k)^*(j-n_{k-1})\le \Big(\frac{j}{j-n_{k-1}}\Big)^{1/p}(j-n_{k-1})^{1/p}(\widetilde{v}_k)^*(j-n_{k-1})\\
&\le 2^{1/p}\|(\widetilde{v}_k)^*\|_{p,\infty}\le  2^{1/p}D_{p}\|(\widetilde{v}_k)^*\|_{p}=2^{1/p}D_{p}\|\widetilde{v}_k\|_{p}\\
&\le 2^{1/p}D_{p}\Big(\|v_k\|_{p}+\frac{1}{(n_k-n_{k-1})^{1/p}}\|\chi_{A_k}\|_{p}\Big)\\
&=2^{1/p}D_{p}\Big(\|v_k\|_{p}+\frac{1}{(n_k-n_{k-1})^{1/p}}(\#A_k)^{1/p}\Big)\\
&\le 2^{1/p}D_{p}\Big(1+\ \frac{(n_k-n_{k-1})^{1/p}}{(n_k-n_{k-1})^{1/p}}\Big)=2^{1/p+1}D_{p}.
\end{align*}
This gives us with \eqref{evoriovg} and \eqref{wdjchcviohivh}
\begin{align*}
& \|\widetilde{z}_N\|_{p,\infty}\lesssim 2^{1+1/p}(1+D_{p}).
\end{align*}
Using \eqref{evhfovho} we conclude that $\|z_N\|_{p,\infty}$ is bounded which proves that the embedding cannot be an isomorphism on $X$ and finishes the proof.
\end{proof}

{\bf{Final note:}} After the main results one can ask a couple of obvious questions:

 Is it possible to obtain the above result  also  for natural embeddings between Lorentz spaces $\ell_{p,r} \hookrightarrow \ell_{q,s}$? 
 
 Can be proven that the embedding (\ref{Lorentz Embd}) is even finitely strictly singular? (In view of  \cite{Pl} in which was proved that the natural embedding $\ell_p \hookrightarrow \ell_q$ is not only strictly singular but also finitely strictly singular this is an interesting question). 
%%%%%%%%%%%%%%%%%%%%%%%%%%%%%%%%

\end{document}